\documentclass[12pt]{amsart}

\usepackage{amssymb,amsmath,latexsym}
\usepackage[varg]{pxfonts}
\usepackage{amsmath, amsthm, amscd, amsfonts, amssymb, graphicx, color}
\usepackage{marvosym}
\usepackage{paralist}
\usepackage{color}
\title[   ]{ Fixed point properties and  $Q$-nonexpansive retractions   in   locally convex spaces}
 \author{ Sompong Dhompongsa$^{1}$,  Poom Kumam$^{1,2,*}$ and Ebrahim  Soori$^{3}$}
 \thanks{ *The corresponding author: poom.kum@kmutt.ac.th  (PK)
  \\ E-mail addresses: sori.e@lu.ac.ir, sori.ebrahim@yahoo.com (ES), sompong.dho@kmutt.ac.th  (SD)  }
\theoremstyle{plain}
\newtheorem{lem}{\textbf{Lemma}}[section]
\newtheorem{thm}[lem]{\textbf{Theorem}}
\newtheorem{co}[lem]{\textbf{Corollary}}

\theoremstyle{definition}

\theoremstyle{definition}
\theoremstyle{remark}

\renewcommand{\sc}{\mathcal{S}}

\newcommand{\ud}{\,\mathrm{d}}
\begin{document}
\begin{large}

\maketitle
%\date{January 1, 2011}
%----------additions
%\dedicatory{To my boss}
%%% ----------------------------------------------------------------------
 \begin{center}
 \begin{normalsize}
% $^{2}$Program in Applied Statistics, Department of Mathematics and Computer Science,\\
%Faculty of Science and Technology, Rajamangala University of Technology Thanyaburi,
%Thanyaburi, Pathumthani 12110, Thailand\\
$^{1}$KMUTT-Fixed Point Theory and Applications Research Group, Theoretical and Computational Science Center (TaCS), Science Laboratory Building, Faculty of Science, King Mongkut's University of Technology Thonburi (KMUTT), \\ 126 Pracha-Uthit Road, Bang Mod, Thung Khru, Bangkok 10140 Thailand.\\
$^{2}$Department of Medical Research, China Medical University Hospital, China Medical University, Taichung 40402, Taiwan\\  $^{3}$Department  of Mathematics, Lorestan University,  P. O. Box 465, Khoramabad, Lorestan, Iran.\\

 \end{normalsize}
 \end{center}
\begin{abstract}

\begin{normalsize}
 Suppose that  $Q$ is a   family of   seminorms on a        locally
convex space $E$ which determines the topology of $E$. We study  the existence of $Q$-nonexpansive retractions  for   families of $Q$-nonexpansive mappings and  prove that     a   separated and sequentially
complete    locally
convex space $E$  that has the weak  fixed point property for
commuting separable semitopological semigroups of $Q$-nonexpansive mappings.  This       proves   the Bruck's problem  \cite{bruckac}     for  locally convex spaces.      Moreover,  we prove  the existence of $Q$-nonexpansive retractions  for  the right
amenable $Q$-nonexpansive semigroups.
\end{normalsize}
\end{abstract}
\begin{normalsize}
   \textbf{keywords}: Weak fixed point property; $Q$-Nonexpansive mapping;   Right amenable   semigroup;  Seminorm; Retraction.
   \end{normalsize}
%%% ----------------------------------------------------------------------
%%% ----------------------------------------------------------------------
%\tableofcontents

\section{ Introduction}
The first nonlinear ergodic theorem for nonexpansive mappings in a Hilbert space was established by Baillon \cite{bia}: Let $C$
be a nonempty closed convex subset of a Hilbert space $H$ and let $T$ be a nonexpansive mapping of $C$ into itself. If the set
$\text{Fix}(T)$ of fixed points of $T$ is nonempty, then for each $x \in C$, the Cesaro means
$
S_{n}x=\frac{1}{n}\sum_{k=1}^{n}T^{k}x
$
converge weakly to some $y \in \text{Fix}(T)$. In Baillon's theorem, putting $y = Px$ for each $x \in C$, $P$ is a nonexpansive retraction of $C$
onto $\text{Fix}(T)$ such that $PT^{n} = T ^{n} P = P$ for all positive integers $n$ and $Px \in \overline{co}\{T ^{n}x: n = 1, 2, . . .\}$ for each $x \in C$. Takahashi \cite{tak1}
proved the existence of such retractions, “ergodic retractions”, for non-commutative semigroups of nonexpansive mappings
in a Hilbert space: If $S$ is an amenable semigroup, $C$ is a closed, convex subset of a Hilbert space $H$ and $\sc= \{T_{s}: s \in S\}$
is a nonexpansive semigroup on $C$ such that $\text{Fix}(\sc)\neq \emptyset$, then there exists a nonexpansive retraction $P$ from $C$ onto $\text{Fix}(\sc)$
such that $PT_{t} = T _{t} P = P$ for each $t \in S$  and $Px \in \overline{co}\,\{T_{t}x:t\in S\}$    for each $x \in C$. These results were extended to uniformly
convex Banach spaces for commutative semigroups in \cite{hi}  and for amenable semigroups in \cite{lst, lnt}. For
some related results, we refer the readers to the works in \cite{said11, said, said2012, said33}. In this paper,   we find some  ergodic retractions  for  locally convex spaces.

Bruck proved in  \cite{bruckac} a Banach space $E$ has the weak fixed point property
for commuting semigroups if it has the weak fixed point property.  In this research,  we prove this problem for  locally
convex spaces.

Let $f$ be a  function of semigroup $S$ into a reflexive Banach space $E$ such that the weak closure of $\{f(t):t\in S\}$ is weakly
compact and let $X$ be a subspace of $B(S)$ containing all the functions $t \rightarrow \left<f(t),x^{\ast}\right>$ with  $ x^{\ast} \in E^{\ast}  $. We know from \cite{hi} that for any $ \mu\in X^{\ast}$, there
exists a unique element $ f_{\mu}$ in $E$ such that $\left<f_{\mu},x^{\ast}\right>=\mu_{t}\left<f(t),x^{\ast}\right>$
for all $ x^{\ast}\in E^{\ast} $. we denote such $f_{\mu}$ by $\int\!f(t)\ud \mu(t)$. Moreover, if $ \mu$ is a mean on $X$ then from \cite{ki}, $\int\! f(t)\ud \mu(t) \in \overline{\rm{co}}\,\{f(t):t\in S\}$. In this paper, we find such  function $ f_{\mu}$  for  locally
convex spaces.
\section{preliminaries }
  The   space of all bounded real-valued functions defined on a semigroup $S $ with supremum
norm is denoted by  $B(S)$.   $l_{t}$ and  $r_{t}$ in $B(S)$  are defined  as follows:  $  (l_{t}g )(s) = g (ts)$ and   $(r_{t}g )(s) = g (st)$,  for all $ s \in S$,  $ t \in S$  and $g\in B(S)$.\\
Suppose that $ X$ is a subspace of $B(S)$ containing 1 and let $ X^{*}$ be its topological  dual space. An element $m $ of $X^{*} $ is said to be a mean on $X$, provided  $\|m\|=m(1)=1$.  For $ m\in X^{*}$ and $g \in X$,   $m_{t}(g(t))$ is often  written instead of $ m(g )$. Suppose that $ X $ is left invariant (respectively, right
invariant), i.e., $ l_{t}(X) \subset X$ (respectively, $ r_{t}(X) \subset X$) for each $ s \in S$. A mean $m$ on $X$ is called    left invariant (respectively, right invariant), provided $m(l_{t}g ) =m(g )$ (respectively, $m(r_{t}g)= m(g )$) for each $t \in S $ and $g \in X$. $X $ is    called      left (respectively, right) amenable if $X$ possesses a left (respectively, right) invariant mean. $X$ is amenable, provided $X$ is both left and right amenable.

% Let $S$ be a semigroup. We denote by $B(S)$ the Banach space of all bounded real-valued functions defined on $ S $ with supremum
%norm. For each $ s \in S$ and $f\in B(S)$ we define $l_{s}$ and  $r_{s}$ in $B(S)$  by\\ \;\indent$\quad (l_{s}f )(t) = f (st)$ ,\qquad $(r_{s}f )(t) = f (ts)$,\quad $ \;\;( t \in S)$.\\
%Let $ X$ be a subspace of $B(S)$ containing 1 and let $ X^{*}$ be its topological dual. An element $\mu $ of $X^{*} $ is said to be a mean on $X$ if $\|\mu\|=\mu(1)=1$. We often write $\mu_{t}(f (t))$ instead of $ \mu(f )$ for $ \mu\in X^{*}$ and $f \in X$. Let $ X $ be left invariant (resp. right
%invariant), i.e. $ l_{s}(X) \subset X$ (resp. $ r_{s}(X) \subset X$) for each $ s \in S$. A mean $\mu$ on $X$ is said to be left invariant (resp. right invariant) if $\mu(l_{s}f ) =\mu(f )$ (resp. $\mu(r_{s}f )= \mu(f )$) for each $s \in S $ and $f \in X$. $X $ is said to be left (resp. right) amenable if $X$ has a left (resp. right) invariant mean. $X$ is amenable if $X$ is both left and right amenable. As is well known, $B(S)$ is amenable when $S$ is a
%commutative semigroup (see  page 29 of \cite{tn}).

%\begin{lem}{\rm(\cite{ma})}.\label{rho} Let $A $ be a strongly positive linear bounded operator on a Hilbert space $ H$ with coefficient $\overline{}$ and $0 < \rho\leq\|A\|^{-1}$ Then $\|I-\rho A\| \leq 1-\rho\,\overline{}$.\/\end{lem}
Recall the following definitions:
\begin{enumerate}
\item A semitopological semigroup is a semigroup $S$ with a Hausdorff topology such that for
each $a \in S$ the mappings $s \rightarrow a.s$ and $s \rightarrow s.a$ from $S$ to $S$ are continuous. For example, $(\mathbb{R} , +)$ with the usual topology on $\mathbb{R}$ and   any semigroup  with the discrete
topology are semitopological semigroups,
\item Suppose that  $Q$ is a   family of   seminorms on a        locally
convex space $E$ which determines the topology of $E$ and  $ C \subset E $.    A mapping $T : C  \rightarrow C$ is said to be
$Q$-nonexpansive provided  the following
inequality holds:
 \begin{equation*}
     q(Tx-Ty) \leq q(x-y),
    \end{equation*}
    for all $x, y \in C$ and  $q \in Q$,

\item
Let $C$ be a nonempty closed and convex subset of $E$. Then, a family   $ \sc=\{T_{s}:s\in S\} $ of mappings from $C$ into itself is said to be a  representation of $S$ as $Q$-nonexpansive mapping on $C$ into itself if $ \sc$ satisfies the following properties:
\begin{enumerate}
  \item [(1)] $T_{st}x=T_{s}T_{t}x$ for all $s, t \in S$ and $x\in C$;
  \item [(2)] for every $s \in S$ the  mapping $T_{s}: C \rightarrow C$ is $Q$-nonexpansive.
\end{enumerate}
 We denote by Fix($ \sc$)
the set of common fixed points of $\sc$, that is \\ Fix($ \sc$)=$\bigcap_{ s\in S}\{x\in C: T_{s}x=x  \}$,
\item   Let  $E$ be a Hausdorff  locally convex space. Then $E$ is
quasi-complete if every bounded Cauchy net is convergent. Observe that \\
 complete $\Rightarrow$   quasi-complete $\Rightarrow$ \text{sequentially complete},

    \item Let   $Q$ be a   family of   seminorms on a  locally convex space $E$ which determines the topology of $E$ and suppose $E$ is Hausdorff and sequentially complete. Consider a  sequence $\{x_{n}\}$  and a    series $\sum_{n=1}^{\infty}x_{n}$  in $E$. A series
  is called  $Q$-absolutely convergent if   $\sum_{n=1}^{\infty}q(x_{n}) < \infty$  for each  $q \in Q$,
\item  Suppose that  $Q$ is a   family of   seminorms on a      locally convex space $E$ which determines the topology of $E$ and  $ C  $   a closed, convex subset of   $E$.   We say that $E$ has the weakly fixed
point property if, for every nonempty  weakly compact convex subset $C$ of $E$, every
$Q$-nonexpansive mapping of $C$ into itself has a fixed point,

%\item let $ S $ be a semigroup and  $ C $ be a closed, convex subset of a
  %Banach space $  E$. We call a representation    $ \sc=\{T_{s}:s\in S\} $   an   asymptotically nonexpansive representation      if
% \begin{equation}\label{asymnep}
    %  \limsup _{n\rightarrow \infty} \|T^{n}_{t}x-T^{n}_{t}y\| \leq \|x-y\|,
   % \end{equation}
 %for all   $x,y \in C$ and $t \in S $.
\item Let $ S $ be a semigroup and  $ C $ be a closed, convex subset of a
  locally convex space $E$. We call a representation    $ \sc=\{T_{s}:s\in S\} $   an  $Q$-nonexpansive representation      if
 \begin{equation*}
     q(T_{t}x-T_{t}y) \leq q(x-y),
    \end{equation*}
 for all   $x,y \in C$ and $t \in S $.
 \item Let $C$ be a nonempty subset of a topological space $X$ and $D$ a nonempty
subset of $C$. Then a continuous mapping $P:C \rightarrow D$ is said to be a retraction
if $Px = x$ for all $x \in D$, i.e., $P^{2} = P$. In such case, $D$ is said to be a retract
of $C$.
\item  Suppose that  $Q$ is a   family of   seminorms on a        locally
convex space $E$ which determines the topology of $E$.  Then  $E$ is separated if and only if   $Q$
  possesses the following property:\\
for every $x \in X  \setminus \{0\}$ there is a seminorm  $q \in Q$ such that $q(x) \neq 0$.
\item
 Suppose that  $Q$ is a   family of   seminorms on a        locally
convex space $X$ which determines the topology of $X$ and $q \in Q$.   From page 3 in \cite{coop}, a linear functional $f: X \rightarrow \mathbb{ R}$ is continues if there are a constant $M \geq 0$ and $q_{1}, \ldots, q_{n} \in Q$ such that $|f(x)| \leq M \displaystyle \max _{1\leq i \leq n} q_{i}(x)$  for all $x \in X$,
 \item let $ S $ be a semigroup.  Suppose that  $Q$ is a   family of   seminorms on a        locally
convex space $E$ which determines the topology of $E$  and  let $ C $ be a closed, convex subset of a
  locally convex space $  E$ and      $ \sc=\{T_{s}:s\in S\} $ a  $Q$-nonexpansive representation  on   $C$. A point $a\in E$ is an attractive point of $\sc$ if $q(a-T_{s}x) \leq q(a-x)$ for all $x \in C$, $s \in S$ and $q \in Q$.
  \item Let $X$ be a vector space over $\mathbb{R}$ or $\mathbb{C}$. A subset $B \subseteq X$ is
balanced if for all  $x \in B$ and  $\alpha$ in the base field such that  $|\alpha | \leq 1$, we have $\alpha x \in B$.
\end{enumerate}
\section{\textbf{Some applications of Hahn Banach theorem in Locally convex spaces}}

     % Then a multivalued mapping $J_{q} : X \rightarrow 2^{X^{*}_{q}}$
% is said to be a   $q$-duality mapping  if $$J_{q}x =\{ j \in X^{*}_{q}: \langle x , j\rangle ={q(x)}^{2}={q^{*}(j)}^{2} \},$$   $J_{q}x \neq \emptyset$ for each $x \in X$, because from Theorem \ref{hahn2}, there exists    a linear functional $f \in  {X^{*}_{q}}$ such that $q^{*}(f)=1$ and  $\langle x , f \rangle=q(x)$ for each $x \in X$. Putting $j=q(x)f$, we have $$\langle x , j\rangle =\langle x , q(x)f \rangle=q(x)\langle x ,  f \rangle={q(x)}^{2},$$ and we have
 %\begin{align*}
 % q^{*}(j)=&\sup\{|j(y)|: y \in Y, q(y)\leq 1\}=\sup\{|q(x)f(y)|: y \in Y, q(y)\leq 1\}\\=&q(x)\sup\{|f(y)|: y \in Y, q(y)\leq 1\}=q(x)q^{*}(f)=q(x)
% \end{align*
Let $Y$ be a    subset  of a  locally
convex space $X$, we put   $q^{*}_{Y}(f)=\sup\{|f(y)|: y \in Y, q(y)\leq 1\}$ and      $q^{*}(f)=\sup\{|f(x)|: x \in X, q(x)\leq 1\}$, for every linear function $f$ on $X$. We will make use of the following Theorems.
\begin{thm}\label{hahn1}
  Suppose that  $Q$ is a   family of   seminorms on a   real     locally
convex space $X$ which determines the topology of $X$ and $q \in Q$ is a continuous  seminorm and $Y$ is a  vector subspace of $X$ such that $Y \cap \{x \in X: q(x)=0\}=\{0\}$.     Let $f$ be a real    linear functional on $Y$ such that   $q^{*}_{Y}(f)< \infty$. Then there exists a continuous linear functional $h$ on $X$ that extends $f$ such that $q^{*}_{Y}(f)=q^{*}(h)$.
\end{thm}
\begin{proof}
      If we define $p: X\rightarrow \mathbb{R}$ by $p(x)=q^{*}_{Y}(f)q(x)$ for each $x \in X$, then we have $p$ is a seminorm on $X$ such that $f(x)\leq p(x)$, for each $x\in Y$. Because,  if $x = 0$, clearly $f(x)=0$ and $0\leq p(x)$. On the other hand, if $x \in Y$ and $x \neq 0$ then from our assumption,  $q(x)\neq 0$ and  $q(\frac{x}{q(x)})= 1$. Therefore, we have $f(\frac{x}{q(x)}) \leq q^{*}_{Y}(f) $, then   $f(x) \leq q^{*}_{Y}(f)  q(x) = p(x)$. Since  $q $ is   continuous,    $p $ is also a continuous  seminorm, therefore    by the Hahn-Banach theorem (Theorem 3.9 in \cite{sob}), there exists a linear continuous extension $h$ of $f$ to $X$ that $h(x)\leq p(x)$ for each $x \in  X$. Hence, since $X$ is a vector space, we have
  \begin{equation}\label{qq}
    |h(x)|\leq q^{*}_{Y}(f)q(x), (x \in X)
  \end{equation}
   and  hence,  $q^{*}(h) \leq q^{*}_{Y}(f)$. Moreover, since $q^{*}_{Y}(f)=\sup \{|f(y)|: q(y)\leq 1\}\leq \sup \{|h(x)|: q(x)\leq 1\}=q^{*}(h) $, we have  $q^{*}_{Y}(f)=q^{*}(h)$.
\end{proof}
\begin{thm}\label{hahn2}
   Suppose that  $Q$ is a   family of   seminorms on a   real     locally
convex space $X$ which determines the topology of $X$ and  $q \in Q$   a nonzero  continuous seminorm.    Let $x_{0}$ be a point in $X$.  Then there exists  a continuous linear functional on $X$ such that $q^{*}(f)=1$ and  $f(x_{0})=q(x_{0})$.
\end{thm}
\begin{proof} Let  $Y:=\{y\in X: q(y)=0\} $.  We   consider two cases:\\
Case 1.   Let $x_{0}\in Y$. \\ Since  $q$  is    continuous,         $Y$ is a closed subset of $X$. Indeed, if  $x \in \overline{Y}$ and $x_{\alpha}\in Y$  is a net  such that $x_{\alpha}\rightarrow x$. Then we have $q(x) =\lim  q(x_{\alpha})= 0$, hence $x \in Y$,  then  $Y$ is a closed.  Let  $ y_{0}$ be a point in $   X \setminus Y$. There exists some $r>0$ such that $q(y-y_{0})>r$ for all $y \in Y$. Suppose that $Z=\{y+\alpha y_{0}: \alpha \in \mathbb{R}, y \in Y\}$, the vector subspace generated by $Y$ and $y_{0}$. Then we define $h:Z\rightarrow \mathbb{R}$ by $h(y+\alpha y_{0})=\alpha$. Obviously,  $h$ is   linear and we  have also
\begin{align*}
  r|h(y+\alpha y_{0})|=r|\alpha|< |\alpha|q(\alpha^{-1}y+y_{0})=q(y+\alpha y_{0})
\end{align*}
for all $y \in Y$ and   $\alpha \in \mathbb{R}$. Therefore $h$ is a   linear functional  on $Z$ that $q^{*}_{Z}(h)$  dose not exceed $r^{-1}$. Putting $p=r^{-1}q$,  we have    $p $ is   a continuous  seminorm such that $h(z)\leq p(z)$ for each $z \in Z$, therefore    by the Hahn-Banach theorem (Theorem 3.9 in \cite{sob}), there exists a linear continuous extension $L$ of $h$ to $X$ that $L(x)\leq p(x)$ for each $x \in  X$.
 We have also $L(x_{0})=h(x_{0})=q(x_{0})=0$. Now, since $q^{*}_{Z}(h)\neq 0$, we have also $q^{*}(L)\neq 0$,  we can  define    $f:=\frac{L}{q^{*}(L)} $. Hence, $f$ is a   linear continuous functional  on $Z$ that $f(x_{0})=q(x_{0})=0$ and also $q^{*}(f)=1$.

Case 2. Let $x_{0}\notin Y$.  \\Let $Z:=\{\alpha x_{0}   :\alpha \in \mathbb{R}\}$ that is the vector subspace generated by $x_{0}$. If we define $h(\alpha x_{0} )=\alpha q(x_{0} )$ then $h$ is a   linear functional  on $Z$ that  $h(x_{0})=q(x_{0})$ and also $q^{*}_{Z}(h)=1$.
 Since $Z \cap Y=\{0\} $, from Theorem \ref{hahn1}, there exists a continuous linear extension $f$ of $h$ to $X$ such that $q^{*}(f)=q^{*}_{Z}(h)=1$.  Obviously, $f(x_{0})=q(x_{0})$.
\end{proof}

%\begin{thm}\label{golizadam}
 %Suppose that  $Q$ is a   family of   seminorms on a   real separated    locally
%convex space $E$ which determines the topology of $E$. Then the weak topology on  $E$ is Hausdorff.
%\end{thm}
%\begin{proof}
%Using Proposition 1.3 in \cite{sob},  let $\{x_{\alpha}\}$ be a net such that  $ x_{\alpha} \rightharpoonup x_{0}$ and $ x_{\alpha} \rightharpoonup y_{0}$ in $E$ with respect to the weak topology. We must show that $x_{0}=y_{0}$.  But we have  $ f(x_{\alpha}) \rightarrow f(x_{0})$ and $ f(x_{\alpha}) \rightarrow f(y_{0})$   for each $f \in E^{*}$ hence $f(x_{0}-y_{0})=0$ for each $f \in E^{*}$. Therefore  from Theorem \ref{hahn2} we have that   for  any  nonzero   seminorm $q$ on $X$  there exists  a continuous linear functional on $X$ such that    $0=f(x_{0}-y_{0})=q(x_{0}-y_{0})$. Because   $E$ is separated we have   $x_{0}=y_{0}$.  Then from  Proposition 1.3 in \cite{sob}, the weak topology on  $E$ is Hausdorff.

%\end{proof}
\section{\textbf{Ergodic retractions for families of $Q$-nonexpansive mappings}}
Let  $Q$ be a   family of   seminorms on a    locally convex space $E$ which determines the topology of $E$. In this section, we study the existence of $Q$-nonexpansive retractions onto the set of common fixed points of a family of
$Q$-nonexpansive mappings that commute with the mappings. A $Q$-nonexpansive retraction that commutes with the mappings is usually called an ergodic retraction.

First, we prove the following theorem  that extends and generalizes Theorem 2.1 in \cite{said2012} which is the main result of this section and  will be essential in the sequel.
\begin{thm}\label{shlp}
Suppose that  $Q$ is a   family of   seminorms on a    locally convex space $E$ which determines the topology of $E$. Let $ C $   be   a  locally weakly compact and  weakly closed convex     subset of $E$.
     Suppose that $\sc = \{T_{i} : i \in I \}$ is a family of $Q$-nonexpansive mappings on $C$ such that
$\text{Fix}(\sc)\neq \emptyset$. Consider the following assumption:
\begin{enumerate}
  \item [(a)] $E$ is separated and for every nonempty  weakly compact convex $\sc $-invariant subset $K$ of $C$,
$K \cap \text{Fix}(\sc) \neq \emptyset$,
  \item [(b)] if $x,y \in C$ and $q(\frac{x+y}{2})=q(x)=q(y)$ for all $q \in Q$, then $x=y$.
\end{enumerate}
Then, for each $i \in I$, there exists a $Q$-nonexpansive retraction $P_{i}$ from $C$ onto
$\text{Fix}(\sc)$, such that $P_{i}T_{i} = T_{i}P_{i} = P_{i}$  and every weakly closed convex $\sc$-invariant
subset of $C$ is also $P_{i}$-invariant.
\end{thm}
\begin{proof}

 Let  $C^{C}$ be the product space with  product topology induced by the  weak topology  on  $C$. Now for a fixed $\alpha \in I$, consider the following set

  $\mathfrak{R}=\{T \in C^{C}: T \; \text{is $Q$-nonexpansive}, T\circ T_{\alpha}=T\\
  \text{and every  weakly closed convex $\sc$-invariant subset of $C$ is also $T$ -invariant}\}$. From the fact that  for each $z \in \text{Fix}(\sc)$, the singleton set  $\{z\}$ is a  weakly closed convex $\sc$-invariant subset of $C$, then for each    $T \in \mathfrak{R}$, $Tz=z$. Fix $z_{0} \in \text{Fix}(\sc)$ and let, for each $x \in C$, $$C_{x}:=\{y \in C : q(y-z_{0}) \leq q(x-z_{0}),\; \text{for all}\; q \in Q\}. $$  For all   $x \in C$ and   $T \in \mathfrak{R}$, we have that  $T(x) \in  C_{x}$ since  $q(T(x)-z_{0}) = q(T(x)-T(z_{0}))\leq q(x-z_{0})$  for all  $q \in Q$. Hence $\mathfrak{R} \subseteq \prod_{x \in C} C_{x}$, where  $\prod_{x \in C} C_{x}$ is the Cartesian product of sets $C_{x}$ for all  $x \in C$. For each $q \in Q$, from  the fact that for each  real number $\lambda$ the level  set $\{x \in C: q(x) \leq \lambda\}$  is   closed and  convex   then from  Corollary 1.5 (p. 126) in \cite{conway}  is   weakly closed, hence  by   Proposition 2.5.2 in \cite{Ag}
each seminorm $q\in Q$ is weakly lower semicontinuous. Because    $ C$ is a    weakly closed convex subset of $E$,  $C_{x}$
is convex and weakly closed for each  $x \in C$. Since $C$ is locally weakly compact and $C_{x}$ is $\tau_{Q}$-bounded, we can conclude that     $C_{x}$ is  weakly compact. By Tychonoff's theorem, we know that  when $C_{x}$ is given the weak topology and $ \prod_{x \in C} C_{x}$ is given the corresponding product topology, $ \prod_{x \in C} C_{x}$ is   compact. Next we prove that  $\mathfrak{R}$ is   closed in $\prod_{x \in C} C_{x}$. Let $\{ T_{\lambda}: \lambda \in \Lambda\}$ be a net in $\mathfrak{R}$ which  converges to $T_{0}$ in  $ \prod_{x \in C} C_{x}$. Hence if $z \in \text{Fix}(\sc)$, $T_{\lambda}z=z$, then  $T_{0}z= weak-\lim_{\lambda}T_{\lambda}(z)=z$. Since  each seminorm $q\in Q$ is weakly lower semicontinuous, we have $q(T_{0}x-T_{0}y) \leq \liminf_{\lambda} q(T_{\lambda}x-T_{\lambda}y)\leq  q(x-y)$, for each $x,y \in C$ and $q\in Q$. Obviously, we have  $T_{0}\circ T_{\alpha}=T_{0}$  and every weakly closed convex $\sc$-invariant
subset of $C$ is also $T_{0}$-invariant. Therefore, $T_{0} \in \mathfrak{R}$. Then $\mathfrak{R}$ is closed in $\prod_{x \in C} C_{x}$.
Since $\prod_{x \in C} C_{x}$ is compact, hence   $\mathfrak{R}$ is compact. Moreover,  $\mathfrak{R}\neq \emptyset$. Indeed, if we define the mapping $S_{n}= \frac{1}{n} \sum _{0}^{n-1}T_{\alpha}^{i} \in \prod_{x \in C} C_{x}$. Then we have,
\begin{align}\label{snkho}
\lim_{n \rightarrow \infty}S_{n}T_{\alpha}-S_{n}=0,
\end{align}
 on $C$. Indeed  for each $q\in Q$, we have
\begin{align*}
 q\big(S_{n}T_{\alpha}(z)-S_{n}(z)\big)&= \frac{1}{n} q\big(T_{\alpha}^{n}(z)-z\big)= \frac{1}{n}q\big(T_{\alpha}^{n}(z)-T_{\alpha}^{n}(z_{0})
+T_{\alpha}^{n}(z_{0}) -z\big)\\&\leq  \frac{1}{n}q\big(T_{\alpha}^{n}(z)-T_{\alpha}^{n}(z_{0})\big)+ \frac{1}{n}q(
z_{0}  -z) \\&\leq  \frac{2}{n}q(
z_{0}  -z) \rightarrow 0,
\end{align*}
 as $n \rightarrow \infty$, for all $z \in C$ and from the fact that  $\prod_{x \in C} C_{x}$ is compact, there exists a ( weakly pointwise) convergent subnet $\{S_{n _{(\lambda)}}\}_{\lambda}$. Hence we can define  $T(x)=weak-\displaystyle\lim_{\lambda}S_{n _{(\lambda)}}(x)$. Next, we will show that $T \in  \mathfrak{R}$, because     each  seminorm $q\in Q$ is weakly lower semicontinuous and $S_{n _{(\lambda)}}$ is  $Q$-nonexpansive for each
$ \lambda $ then  $T$ is $Q$-nonexpansive. From \eqref{snkho}, we also   have $T(T_{\alpha}x)=\displaystyle weak-\lim_{\lambda}S_{n _{(\lambda)}}(T_{\alpha}x)= \displaystyle weak-\lim_{\lambda}S_{n _{(\lambda)}}x=Tx$. Finally, every  weakly closed convex $\sc$-invariant subset of $C$ is  $S_{n }$-invariant and hence  is $T$ -invariant. Then $T \in \mathfrak{R}\neq \emptyset$.

Now let us to define a preorder $\preceq$ in $\mathfrak{R}$ by $T \preceq U$ if $q(Tx-Ty)\leq q(Ux-Uy)$ for each $x,y \in C$ and  $q\in Q$, and by  using a method similar to Bruck's method \cite{broook}, we find a minimal element $T_{min}$ in  $\mathfrak{R}$. Indeed, Using Zorn's Lemma, it is enough that we   show that each  linearly ordered subset of $\mathfrak{R}$ has
a lower bound in  $\mathfrak{R}$. Hence, let $\{ A_{\lambda}\}$ be a linearly ordered subset of $\mathfrak{R}$. Then the family of sets  $\{T \in \mathfrak{R}:  T \preceq A_{\lambda}\}$ is a linearly ordered  subset of  $\mathfrak{R}$  by inclusion. Taking  into account  the  closeness proof of  $\mathfrak{R}$  in
 $\prod_{x \in C} C_{x}$, these sets also   can accordingly be closed in  $\mathfrak{R}$,  and
hence   compact.
  Then from the finite intersection property, there exists  $R \in\bigcap_{\lambda}\{T \in \mathfrak{R}:  T \preceq A_{\lambda}\}$ with $R \preceq A_{\lambda}$ for all $\lambda$. Then each  linearly ordered subset of $\mathfrak{R}$ has
a lower bound in  $\mathfrak{R}$. We have
shown  until now that there exist    a minimal element $P_{\alpha}$  in the following sense:

 \text{if }$ T \in \mathfrak{R} \; \text{and } q(Tx-Ty)\leq q(P_{\alpha}x-P_{\alpha}y)$ \text{ for each  }  $x,y \in C $ \text{and } $q \in Q$ \\
  \text{then } $q(Tx-Ty)= q(P_{\alpha}x-P_{\alpha}y)$. \quad ($*$)

  Next we prove that $P_{\alpha}x \in \text{Fix}(\sc)$  for every    $x  \in C $. For a given   $x  \in C $, consider $K:=\{T(P_{\alpha}x): T \in \mathfrak{R}\}$.
   From the fact that  $\mathfrak{R}$  is convex and compact, by   Proposition 3.3.18 and the Definition 3.3.19 in  \cite{runde}, we conclude that    $K$ is a nonempty  weakly compact convex subset of $C$. Now we have $S(K)\subset K$ for each $S \in \sc$, because  $STT_{\alpha}=ST$  for each  $T \in \mathfrak{R}$ hence $ST \in \mathfrak{R}$ i.e,  $K$ is $\sc$-invariant.

    First, consider the case (a): from our assumption $K \cap \text{Fix}(\sc) \neq \emptyset$. Then there exists $L \in \mathfrak{R}$
such that  $L(P_{\alpha}x) \in \text{Fix}(\sc)$. Suppose that  $y=L(P_{\alpha}x)$. Since $P_{\alpha}, L \in \mathfrak{R}$ and  the set  $\{y\}$ is $\sc$-invariant, so we have $P_{\alpha}(y)=L(y)=y$, and since  $P_{\alpha}$ is minimal, we have $q(P_{\alpha}x-y)= q(P_{\alpha}x-P_{\alpha}y)= q\big(L(P_{\alpha}x)-L(P_{\alpha}y)\big)= q\big(L(P_{\alpha}x)-y\big)=0$, for each $q \in Q$ and since  $E$ is separated, $ P_{\alpha}x-y=0$, hence  $P_{\alpha}x=y \in  \text{Fix}(\sc)$ and this   holds  for each $x \in C$.

Now, consider the case (b): because $P_{\alpha} \in \mathfrak{R}$ we have  $T_{i}P_{\alpha} T_{\alpha}=T_{i}P_{\alpha}$ for all $i$. Therefore it is easy to verify that  $T_{i}P_{\alpha}\in \mathfrak{R}$,  for all $i$. Since  $\mathfrak{R}$  is convex, using ($*$), we have
 $$q(T_{i}P_{\alpha}x-z)= q\big(T_{i}P_{\alpha}x-T_{i}P_{\alpha}z\big)= q(P_{\alpha}x-z)= q \big(\frac{T_{i}P_{\alpha}x+P_{\alpha}x}{2}-z \big),$$
for each $x \in C, z \in \text{Fix}(\sc)$ and $i \in I$. Then from our assumption we have $T_{i}P_{\alpha}x=P_{\alpha}x$ for each $x \in C$ and $i \in I$. Hence, $P_{\alpha}x  \in  \text{Fix}(\sc)$,  for each $x \in C$.

 Since     $P_{\alpha} \in \mathfrak{R}$  and  $\{P_{\alpha}x\}$ is $\sc$-invariant  for each $x \in C$,  we conclude that  $P_{\alpha}^{2}=P_{\alpha}$ and $P_{\alpha}T_{\alpha}=T_{\alpha}P_{\alpha}=P_{\alpha} $.
\end{proof}
As a consequence of Theorem \ref{shlp}, we prove how we obtain an ergodic retraction by a $Q$-nonexpansive retraction.
\begin{co}\label{kgfdy}
Suppose that  $Q$ is a   family of   seminorms on a    locally
convex space $E$ which determines the topology of $E$. Let $ C $   be   a  locally weakly compact and  weakly closed convex     subset of $E$.
     Suppose that $\sc = \{T_{i} : i \in I \}$ is a family of $Q$-nonexpansive mappings on $C$ such that
$\text{Fix}(\sc)\neq \emptyset$. If there is a $Q$-nonexpansive retraction $R$ from $C$ onto
$\text{Fix}(\sc)$,
then for each $i \in I$, there exists a $Q$-nonexpansive retraction $P_{i}$ from $C$ onto
$\text{Fix}(\sc)$, such that $P_{i}T_{i} = T_{i}P_{i} = P_{i}$, and every weakly closed convex $\sc \cup \{R\}$-invariant
subset of $C$ is also $P_{i}$-invariant.
\end{co}
\begin{proof} Setting  $ \sc ^{'}:= \sc \cup \{R\}$ and \\
 $\mathfrak{R^{'}}=\{T \in C^{C}: T \; \text{is $Q$-nonexpansive}, T\circ T_{\alpha}=T\\
  \text{and every  weakly closed convex $\sc ^{'}$-invariant subset of $C$ is also $T$ -invariant}\}$,
 we get that $\text{Fix}(\sc^{'})= \text{Fix}(\sc )$ and by replacing $\sc$ with $ \sc ^{'}$  and $\mathfrak{R}$ with $\mathfrak{R^{'}}$ in the proof of Theorem \ref{shlp}, we find a minimal element  $P_{\alpha}$ in the sense of ($*$). Now we have  $R\circ T \in \mathfrak{R^{'}}$ for each $T \in \mathfrak{R^{'}}$. Indeed,  $R\circ T\circ T_{\alpha}=R\circ T$  for each $T \in \mathfrak{R^{'}}$ and because $R \in \sc ^{'}$, we have that every  weakly closed convex $\sc ^{'}$-invariant subset of $C$ is also $R$ -invariant, therefore is $R\circ T$-invariant for each $T \in \mathfrak{R^{'}}$. Hence  for each $x \in C$, the set $K=\{T(P_{\alpha}x): T \in \mathfrak{R^{'}}\}$ is an $R$-invariant subset of $C$  for each $T \in \mathfrak{R^{'}}$. Therefore from the fact that $R(K) \subset K \cap   R(C)=K \cap \text{Fix}(\sc)$, we have $\text{Fix}(\sc) \neq \emptyset$. Now by repeating the reasoning used in Theorem \ref{shlp}, we will get the desired result.
\end{proof}
As an application of Theorem  \ref{kgfdy}, we have the following Theorem:
\begin{co}\label{bwopxh}
Suppose that  $Q$ is a   family of   seminorms on a    locally
convex space $E$ which determines the topology of $E$. Let $ C $   be   a  locally weakly compact and  weakly closed convex     subset of $E$.
     Suppose that $\sc = \{T_{i} : i \in I \}$ is a family of $Q$-nonexpansive mappings on $C$ such that
$\text{Fix}(\sc)\neq \emptyset$. Consider the following assumption:
\begin{enumerate}
  \item [(a)] $E$ is separated and for every nonempty  weakly compact convex $\sc $-invariant subset $K$ of $C$,
$K \cap \text{Fix}(\sc) \neq \emptyset$,
  \item [(b)] if $x,y \in C$ and $q(\frac{x+y}{2})=q(x)=q(y)$ for all $q \in Q$, then $x=y$,
 \item [(c)]   there exists  a $Q$-nonexpansive retraction $R$ from $C$ onto
$\text{Fix}(\sc)$.
\end{enumerate}
Let   $\{P_{i}\}_{i \in I}$  be the family of retractions obtained in the above Theorem. Then for each $x \in C$, $$\overline{\{ T_{i}^{n}x: i \in I, n \in \mathbb{N}\}}^{\tau_{Q}} \cap \text{Fix}(\sc)\subseteq \overline{\{ P_{i}(x): i \in I\}}^{\tau_{Q}}.$$
\end{co}
\begin{proof}
Consider an $\epsilon > 0$ and let  $g \in \overline{\{ T_{i}^{n}x: i \in I, n \in \mathbb{N}\}}^{\tau_{Q}} \cap \text{Fix}(\sc)$. Then for each $p \in Q$,  there exists $i \in I$ and $ n \in \mathbb{N}$ such that $q( T_{i}^{n}x -g)< \epsilon$. From our  assumptions  and using Theorems \ref{shlp}  and \ref{kgfdy}, there exists a $Q$-nonexpansive retraction $P_{i}$ such that $P_{i}=P_{i}T_{i}$, hence we have,
$$q( P_{i}x -g)= q( P_{i}T_{i}^{n}x -g)\leq q( T_{i}^{n}x -g)< \epsilon,$$
then we conclude $g \in \overline{\{ P_{i}(x): i \in I\}}^{\tau_{Q}}$.
\end{proof}

\section{The  weakly  fixed point property for commuting semigroups}
Let  $Q$ be a   family of   seminorms on a      locally convex space $E$ which determines the topology of $E$. The   goal   of this section  is to show If $E$ has the weakly fixed point property, then $E$ has the weakly fixed
point property for commuting separable semitopological semigroups.

We will   make use of the following two Theorems in this section.
\begin{thm}\label{bruuks}
 Suppose that  $Q$ is a   family of   seminorms on a      locally convex space $E$ which determines the topology of $E$. Let $E$ be Hausdorff and sequentially complete and    $C$ be a nonempty  closed convex and bounded subset of   $E$. If   $\{F_{n}\}$
is a descending sequence of nonempty $Q$-nonexpansive retracts of $C$, then $\bigcap_{n=1}^{\infty}  F_{n}$ is
the fixed point set of some $Q$-nonexpansive mapping  $r: C \rightarrow C$.
\end{thm}
\begin{proof}
For each  $n \in \mathbb{N}$, let $r_{n}$ be a $Q$-nonexpansive retraction from $C$ onto $F_{n}$. Let $\{\lambda_{n}\}_{n=1}^{\infty}$ be a sequence such that $\lambda_{n} > 0$ for each $n \in \mathbb{N}$ and  $\sum_{n=1}^{\infty} \lambda_{n}=1$ and
\begin{equation}\label{fcdsgr}
  \lim_{n} \frac{\sum_{j=n+1}^{\infty} \lambda_{j}}{\sum _{j=n}^{\infty}\lambda_{j}}=0.
\end{equation}
 From our assumptions and as in Lemma 1 in \cite{bruckac}, we may assume that  $r=\sum_{n=1}^{\infty} \lambda_{n}r_{n}$. Then by   lower
semicontinuity of each $q \in Q$, we have that  $r: C \rightarrow C$ is a  $Q$-nonexpansive mapping such that $\bigcap_{n=1}^{\infty}  F_{n} \subseteq \text{Fix}(r)$. Now it suffices to show that $\text{Fix}(r)  \subseteq \bigcap_{n=1}^{\infty}  F_{n}$. Hence consider  $x \in \text{Fix}(r)$. Then by   lower
semicontinuity of each $q \in Q$, we have
\begin{align}\label{wewsdr}
 q\big(x-r_{n}(x)\big)=&  q\big(r(x)-r_{n}(x)\big)= q\big(\sum _{j=1}^{\infty}\lambda_{j}[r_{j}(x)-r_{n}(x)]\big)\nonumber \\ \leq &   \sum _{j=1}^{\infty}\lambda_{j}q(r_{j}(x)-r_{n}(x)),
\end{align}
for  each $q \in Q$. Because for $1\leq j< n$, $r_{n}(x) \in F_{n} \subseteq F_{j}$, then we have $r_{j}r_{n}(x)=r_{n}(x)$. Therefore we have
\begin{equation*}
  q(r_{j}(x)-r_{n}(x))= q(r_{j}(x)-r_{j}r_{n}(x)) \leq q\big(x-r_{n}(x)\big);
\end{equation*}
now for $j=n$, $ q(r_{j}(x)-r_{n}(x))= 0$;
finally let $$d_{q}= q-diam C=  \sup\{q(x-y): x,y \in C\},$$   then for $j> n$, $q(r_{j}(x)-r_{n}(x)) \leq d_{q}$. Then from \eqref{wewsdr} we have
\begin{align*}
 q\big(x-r_{n}(x)\big) \leq &   \sum _{j=1}^{n-1}\lambda_{j}q(x-r_{n}(x)) + d_{q}\sum _{j=n+1}^{\infty}\lambda_{j}.
\end{align*}
        Because $\sum_{n=1}^{\infty} \lambda_{n}=1$, this in turn concludes that
\begin{align*}
 q\big(x-r_{n}(x)\big) \leq  d_{q}\frac{\sum_{j=n+1}^{\infty} \lambda_{j}}{\sum _{j=n}^{\infty}\lambda_{j}}.
\end{align*}
Then by \eqref{fcdsgr}, $r_{n}(x)\rightarrow x$, in $\tau_{Q}$ when  $n \rightarrow \infty$. But from the fact that $\{F_{n}\}$ is descending,  therefore $r_{n}(x) \in F_{m}$ for each $n \geq m$, and  since $ F_{m}$ is the fixed point set of continuous mapping  $r_{m}$,   $ F_{m}$ is closed in topology $\tau_{Q}$. Now since  $\lim r_{n}(x)=x$, we have $x \in F_{m}$ for $m=1,2,3, \cdots$, therefore $F(r)=\bigcap_{n} F_{n}$.
\end{proof}

\begin{lem} \label{commmuting}
 Suppose that  $Q$ is a   family of   seminorms on a      locally convex space $E$ which determines the topology of $E$.  Let $E$ be separated and sequentially complete.     If $E$   has the weakly fixed point property,  then $E$ has the  weakly fixed point property for commuting sequences of $Q$-nonexpansive mappings on $C$.
\end{lem}
\begin{proof}
 Since  $E$ is separated,    $E$ is Hausdorff with respect to the weak topology.   Let     $C$ be a nonempty  weakly compact  convex subset of   $E$. Since  $C$ is    weakly compact, then  $C$ is weakly closed and  by Proposition 2.7 in \cite{sob},  is weakly bounded hence  by Corollary 3.31 in \cite{sob} is bounded.
 Suppose that $\{T_{n}\}$ be a commuting sequence of $Q$-nonexpansive mappings on $C$. First, we show $\bigcap_{j=1}^{n} \text{Fix}(T_{j})$ is a nonempty $Q$-nonexpansive retract of $C$, for all $n \in \mathbb{N}$. The proof is by induction on $n$. For $n=1$, from Theorem \ref{shlp} and the fact that $E$   has the weakly fixed point property, $\text{Fix}(T_{1})$ is a nonempty $Q$-nonexpansive retract of $C$. Now let  $\bigcap_{j=1}^{n} \text{Fix}(T_{j})$ be a nonempty $Q$-nonexpansive retract of $C$ and $R: C \rightarrow \bigcap_{j=1}^{n} \text{Fix}(T_{j})$ be a   $Q$-nonexpansive retraction. Now we show that  $  \text{Fix}(T_{n+1}R)= \bigcap_{j=1}^{n+1} \text{Fix}(T_{j})$. Obviously we have $ \bigcap_{j=1}^{n+1} \text{Fix}(T_{j}) \subseteq  \text{Fix}(T_{n+1}R)$. For the reverse inclusion, let $x \in  \text{Fix}(T_{n+1}R)$ then $T_{n+1}R( x) =x$. Since from our assumptions $T_{n+1} $ commutes with $T_{1}, \ldots, T_{n}$ and $ R(x) \in \bigcap_{j=1}^{n} \text{Fix}(T_{j})$, we conclude that $\bigcap_{j=1}^{n} \text{Fix}(T_{j})$ is $T_{n+1}$-invariant and hence  $x=T_{n+1}( R(x)) \in \bigcap_{j=1}^{n} \text{Fix}(T_{j})$. Therefore, $R(x)= x $ and so $x=T_{n+1}(R( x)) =T_{n+1}(x)$. Then $x \in \bigcap_{j=1}^{n+1} \text{Fix}(T_{j})$. Thus $\text{Fix}(T_{n+1}R) \subseteq \bigcap_{j=1}^{n+1} \text{Fix}(T_{j})$, so $\text{Fix}(T_{n+1}R) = \bigcap_{j=1}^{n+1} \text{Fix}(T_{j})$. But,  from Theorem \ref{shlp} and the fact that $E$   has the weakly fixed point property, the fixed point set of a $Q$-nonexpansive self mapping of $C$ is  a nonempty $Q$-nonexpansive retract of $C$. Therefore,  $  \bigcap_{j=1}^{n+1} \text{Fix}(T_{j})$ is  a nonempty $Q$-nonexpansive retract of $C$. Until now we have shown that  $  \bigcap_{j=1}^{n} \text{Fix}(T_{j})$ is  a nonempty $Q$-nonexpansive retract of $C$ for each $n \in \mathbb{N}$, thus from Theorem \ref{bruuks}, we conclude that $  \bigcap_{j=1}^{\infty} \text{Fix}(T_{j})$ is
the fixed point set of some $Q$-nonexpansive mapping  $r: C \rightarrow C$. So by       the weakly fixed point property of $E$, we have proved that
$  \bigcap_{j=1}^{\infty} \text{Fix}(T_{j})=\text{Fix}(r)\neq \emptyset$, this   completes the proof.
\end{proof}
Now we prove   the main  conclusion   of this section.
\begin{thm} \label{separable}
 Suppose that  $Q$ is a   family of   seminorms on a      locally convex space $E$ which determines the topology of $E$.  Let $E$ be separated and sequentially complete.  If $E$   has the weakly fixed point property,  then $E$ has the  weakly fixed point property for commuting separable semitopological semigroups.
\end{thm}
\begin{proof}
Let     $C$ be a nonempty  weakly compact  convex subset of   $E$.  Let $ \sc=\{T_{s}:s\in S\} $  be  a  continuous  representation of a commuting separable semitopological semigroups $S$. Suppose that  $\{s_{n}\}$ is a countable  subset of $S$ which is  dense in $S$. Hence $\{T_{s_{n}}\}$ is also a  commuting sequence.  So from Theorem \ref{commmuting} and  the  weakly fixed point property, $\{T_{s_{n}}\}$  has
a common fixed point $z_{0}$ in $C$. But from the fact that  $\{s_{n}\}$ is a countable dense  subset of $S$ and $ \sc $  is  a  continuous  representation of $S$, we conclude that   $\{T_{s}x: s \in S\} \subseteq \overline{\{T_{s_{n}x}\}}^{\tau_{Q}}$ for each $x \in C$. Putting $x=z_{0}$, we have $\{T_{s}z_{0}: s \in S\} \subseteq \overline{\{T_{s_{n}z_{0}}\}}^{\tau_{Q}}=\{z_{0}\}$. Therefore, $z_{0}$ is
a common fixed point for $S$ and so $\text{Fix}(S) \neq \emptyset$ and the proof is
complete.
\end{proof}

\noindent{\bf Open problems:} Theorem \ref{separable} true for amenable semitopological semigroups?

\section{\textbf{Ergodic retractions for semigroups in  locally convex spaces} }

Before going to our ergodic theorem, we will need the following two Theorems.

 \begin{lem}\label{rooh}  Let $ S $ be a semigroup, $E$   be a real   dual  locally convex space  with real predual   locally convex space $D$    and  $U$   a convex neighbourhood  of $0$ in  $D$ and $p_{U}$ be the associated
Minkowski functional.  Let $f : S \rightarrow E$
be a     function  such that  $\langle x , f (t) \rangle \leq 1$  for each  $t \in S$ and $x \in U$. Let $X$ be a subspace of $B(S)$ such that   the mapping $t \rightarrow  \langle   x , f (t)\rangle$ be an element of $ X $, for
each     $ x \in D$.   Then, for any $\mu \in X^{*}$,  there exists a unique element $F_{\mu} \in E$ such that
 $\langle  x , F_{\mu}  \rangle = \mu_{t} \langle   x , f (t)\rangle$ , for all    $ x \in D$.
Furthermore, if $1 \in X$ and $\mu$ is
a mean on $X$, then $F_{\mu}$ is contained in $ {\overline{{\text{co}}\{f(t) : t \in S\}}}^{w^{*}}$.
 \end{lem}
 \begin{proof}
  We define $F_{\mu}$ by $\langle x , F_{\mu} \rangle = \mu_{t} \langle   x , f (t)\rangle$ for all  $ x \in D$. Obviously, $ F_{\mu}$ is
  linear in $x$. Moreover, from  Proposition 3.8  in \cite{sob},  we have
\begin{align}\label{contin}
| \langle x , F_{\mu}\rangle |  = & | \mu_{t} \langle    x , f (t)\rangle |
  \leq   \sup_{t } |\langle    x , f (t)\rangle \cdot\|\mu\|\leq P_{U}(x)\cdot \|\mu\|,
\end{align}
 for all  $ x  \in D$. Let $(x_{\alpha})$ be a net in $  D$ that converges to $ x_{0}$.  Then  by \eqref{contin} we have
 \begin{align*}
  | \langle x_{\alpha} , F_{\mu}\rangle - \langle x_{0} , F_{\mu}\rangle|  = | \langle x_{\alpha}- x_{0} , F_{\mu}\rangle | \leq P_{U}( x_{\alpha}- x_{0})\cdot \|\mu\|,
 \end{align*}
taking limit, since from Theorem 3.7  in \cite{sob},  $P_{U}$  is continuous, we have $F_{\mu}$  is continues on $D$, hence  $F_{\mu} \in E$.

Now, let   $1 \in X$ and $\mu$ be
a mean on $X$. Then, there exists a net $\{\mu_{\alpha}\}$
of finite means on $X$ such that $\{\mu_{\alpha}\}$ converges to $\mu$ with the weak$^{*}$  topology on
$X^{*}$. We may consider that
\begin{align*}
  \mu_{\alpha}=\sum _{i=1}^{n_{\alpha}}\lambda_{\alpha,i}\delta_{t_{\alpha,{i}}}.
\end{align*}
Therefore,
\begin{align*}
  \langle  x ,  F_{\mu_{\alpha}}  \rangle= (\mu_{\alpha})_{t}  \langle  x , f(t) \rangle= \langle x ,   \sum _{i=1}^{n_{\alpha}}\lambda_{\alpha,i}f({t_{\alpha,{i}}})   \rangle, ( \forall x \in D, \forall \alpha),
\end{align*}
then   we have
\begin{align*}
 F_{\mu_{\alpha}}= \sum _{i=1}^{n_{\alpha}}\lambda_{\alpha,i}f({t_{\alpha,{i}}})\in {\text{co}}\{f(t) : t \in S\}, ( \forall \alpha ),
\end{align*}
now since,
\begin{align*}
  \langle  x ,  F_{\mu_{\alpha}} \rangle= (\mu_{\alpha})_{t}  \langle   x , f(t) \rangle \rightarrow \mu_{t}  \langle    x , f(t) \rangle= \langle   x , f(t) \rangle, (x \in D),
\end{align*}
 $\{F_{\mu_{\alpha}}\}$ converges to $F_{\mu}$ in the weak$^{*}$  topology. Hence
\begin{align*}
 F_{\mu}\in \overline{{\text{co}}\{f(t) : t \in S\}}^{w^{*}},
\end{align*}
we can write  $F_{\mu}$ by  $\int f(t) d\mu(t)$.
 \end{proof}

\begin{lem}\label{bshhwj}
Let $ S $ be a semigroup,  $ C $      a closed convex     subset of a real
   locally convex space $E$. Let $\mathcal{B}$ be a
base at 0 for the topology consisting of convex, balanced sets.  Let $Q=\{q_{V}: V \in \mathcal{B} \}$
which      $q_{V}$  is the associated
Minkowski functional with  $V$.   Let $ \sc=\{T_{s}:s\in S\} $  be   a $Q$-nonexpansive representation of $S$ as     mappings from $C$ into
itself  and $X$ be a subspace of $B(S)$ such that $ 1 \in X $ and $\mu$ be a mean on $X$.
If we write $T_{\mu}x $ instead of
$\int\!T_{t}x\, d\mu(t),\/$
 then the following hold.
 \begin{enumerate}
   \item [(i)] If the mapping $t\rightarrow \langle  T_{t}x - T_{t}y , x^{*}  \rangle$ is  an element of $ X $ for
each $t \in S$,  $x, y \in C$ and $ x^{*} \in E^{*}$ then $T_{\mu}$ is a $Q$-nonexpansive mapping from $C$ into $C$,
   \item [(ii)]  if the mapping $t \rightarrow \left<T_{t}x, x^{*}\right>$ is an element of $ X $ for
each $x \in C$ and $ x^{*} \in E^{*}$ then $T_{\mu}x=x $  for each  $x \in Fix(\sc) $,
   \item [(iii)] If moreover $E$   is a real   dual  locally convex space  with real predual   locally convex space $D$  and  $ C $      a $w^{*}$-closed convex     subset  of $E$   and  $U$   a convex neighbourhood  of $0$ in  $D$ and $p_{U}$ is the associated
Minkowski functional.      Let   the mapping $t \rightarrow \langle z , T _{t}x \rangle $ is an element of $ X $ for
each $x \in C$ and $z \in D$ then  $T_{\mu}x \in \overline{co \,\{T_{t}x:t\in S\}}^{w^{*}}$ for each $x\in C$,
   \item[(iv)]  if the mapping $t \rightarrow \left<T _{t}x, x^{*}\right>$ is an element of $ X $ for
each $x \in C$ and $ x^{*} \in E^{*}$ and $  X$ is $ r_{s} $-invariant  for
  each $ s \in S $ and  $ \mu $ is right invariant, then $T_{\mu}T_{t} =T_{\mu} $ for each $ t \in S $,
   \item [(v)]  if $a \in E$ is an $Q$-attractive point of $\sc$ and  the mapping $t\rightarrow \langle a - T_{t}x , x^{*}  \rangle$ is  an element of $ X $ for
each $t \in S$,  $x \in C$ and $ x^{*} \in E^{*}$   then  $a  $ is an $Q$-attractive point of $T_{\mu}$.
 \end{enumerate}
\end{lem}
\begin{proof}
 (i) Let $x, y \in C$ and   $V \in \mathcal{B}$.  By  Proposition 3.33  in \cite{sob}, the topology on $E$ induced by
$Q$ is the original topology on $E$.
 % By page 25 in \cite{rudin}, the seminorms on
%$X$ will turn out to be precisely the Minkowski functionals of balanced
%convex absorbing sets.
By Theorem 3.7  in \cite{sob}, $q_{V}$ is a continuous seminorm and from  Theorem 1.36 in \cite{rudin}, $q_{V}$ is a nonzero seminorm because if $x \notin V$  then $q_{V}(x) \geq 1$, hence   from Theorem \ref{hahn2},  there exists a functional $x^{*}_{V} \in X^{*}$ such that    $q_{V}( T_{\mu}x - T_{\mu}y) = \langle  T_{\mu}x - T_{\mu}y , x^{*}_{V}  \rangle$ and $q_{V}^{*}(x^{*}_{V})=1$, and since    from Theorem 3.7 in  \cite{sob},  $q_{V} (z)\leq 1$ for each $z \in V$, we conclude that  $\langle z ,  x^{*}_{V}\rangle \leq 1  $ for all $z \in V$.  Therefore     from Theorem  3.8 in \cite{sob}, $\langle z , x^{*}_{V}   \rangle \leq  q_{V}(z)$ for all $z \in E$.  Hence from the fact that   the function  $t\rightarrow \langle  T_{t}x - T_{t}y , x^{*}  \rangle$   is  an element of $ X $ for
each $t \in S$ and $ x^{*} \in E^{*}$,  we have
\begin{align*}
q_{V}( T_{\mu}x - T_{\mu}y) =& \langle  T_{\mu}x - T_{\mu}y , x^{*}_{V}  \rangle= \mu_{t} \langle  T_{t}x - T_{t}y , x^{*}_{V}  \rangle
 \\ \leq &\| \mu\| \sup_{t}| \langle  T_{t}x - T_{t}y , x^{*}_{V}  \rangle | \\ \leq & \sup_{t}
q_{V}( T_{t}x - T_{t}y) \\ \leq &  q_{V}(x-y),
\end{align*}
 then  we have
\begin{align*}
q_{V}( T_{\mu}x -  T_{\mu}y) \leq & q_{V}(x-y),
\end{align*}
  for all   $V \in \mathcal{B}$.

(ii)
Let  $x \in Fix(\sc) $ and  $x^{*} \in  E^{*}$. Then we have
\begin{equation*}
   \langle T_{\mu}x, x^{*}  \rangle= \mu_{t} \langle T_{t}x , x^{*}  \rangle= \mu_{t} \langle x , x^{*}  \rangle=  \langle x , x^{*}  \rangle
\end{equation*}

 (iii) this assertion  concludes from  Theorem \ref{rooh}.

   (iv) for  this assertion, note that
\begin{equation*}
   \langle T_{\mu}(T_{s}x) , x^{*}  \rangle= \mu_{t} \langle T_{ts}x , x^{*}  \rangle= \mu_{t} \langle T_{t}x , x^{*}  \rangle=  \langle T_{\mu}x, x^{*}  \rangle,
\end{equation*}
 (v) Let $x \in C$ and   $V \in \mathcal{B}$.
 From Theorem \ref{hahn2},  there exists a functional $x^{*}_{V} \in X^{*}$ such that    $q_{V}(a - T_{\mu}x) = \langle a - T_{\mu}x , x^{*}_{V}  \rangle$ and $q_{V}^{*}(x^{*}_{V})=1$. Since    from Theorem 3.7 in  \cite{sob},  $q_{V} (z)\leq 1$ for each $z \in V$, we conclude that  $\langle z ,  x^{*}_{V}\rangle \leq 1  $ for all $z \in V$.  Therefore     from Theorem  3.8 in \cite{sob}, $\langle z , x^{*}_{V}   \rangle \leq  q_{V}(z)$ for all $z \in E$.  Hence from the fact that   the function  $t\rightarrow \langle a - T_{t}x , x^{*}  \rangle$   is  an element of $ X $ for
each $t \in S$ and $ x^{*} \in E^{*}$,  we have
\begin{align*}
q_{V}(a - T_{\mu}x) =& \langle a - T_{\mu}x , x^{*}_{V}  \rangle= \mu_{t} \langle a - T_{t}x , x^{*}_{V}  \rangle
 \\ \leq &\| \mu\| \sup_{t}| \langle   a - T_{t}x , x^{*}_{V}  \rangle | \\ \leq & \sup_{t}
q_{V}(  a - T_{t}x) \\ \leq &  q_{V}(a-x),
\end{align*}
 then  we have
\begin{align*}
q_{V}( a - T_{\mu}x) \leq & q_{V}(a-x),
\end{align*}
  for all   $V \in \mathcal{B}$.
\end{proof}
Now we exhibit our ergodic theorem.
\begin{thm}\label{vhgfdkkjy}
Let $ S $ be a semigroup,  $ C $      a   locally weakly compact
and   closed convex       subset of a real
   locally convex space $E$. Let $\mathcal{B}$ be a
base at 0 for the topology consisting of convex, balanced sets.  Let $Q=\{q_{V}: V \in \mathcal{B} \}$
which      $q_{V}$  is the associated
Minkowski functional with  $V$.   Let $ \sc=\{T_{s}:s\in S\} $    be a right    amenable   $Q$-nonexpansive semigroup on $C$   such that
$\text{Fix}(\sc)\neq \emptyset$.   Consider the following assumption:
\begin{enumerate}
  \item [(a)] $E$ is separated and for every nonempty  weakly compact convex $\sc $-invariant subset $K$ of $C$,
$K \cap \text{Fix}(\sc) \neq \emptyset$,
  \item [(b)] if $x,y \in C$ and $q_{V}(\frac{x+y}{2})=q_{V}(x)=q_{V}(y)$ for all $V \in \mathcal{B}$, then $x=y$.
\end{enumerate}
Then there exists a $Q$-nonexpansive retraction $P$ from $C$ onto
$\text{Fix}(\sc)$, such that $PT_{t}= T_{t}P = P$ for every $t \in S$, and every weakly closed convex $\sc$-invariant
subset of $C$ is also $P$-invariant.
\end{thm}
\begin{proof}
Consider the following set:

  $\mathfrak{R}=\{T \in C^{C}: T \; \text{is $Q$-nonexpansive}, T\circ T_{t}=T, \forall t \in S\\
  \text{and every  weakly closed convex $\sc$-invariant subset of $C$ is also $T$ -invariant}\}$.

   Replacing this set  by  $\mathfrak{R}$  in the proof of Theorem \ref{shlp},     we can repeat the argument used in the proof of Theorem \ref{shlp} to get
the desired result.  Of Course,    from Theorem \ref{bshhwj},  we have  $T_{\mu} \in \mathfrak{R}\neq \emptyset$ and $C_{x} $ be as in the proof of Theorem \ref{shlp}, then     $\mathfrak{R} \subseteq \prod_{x \in C} C_{x}$.  As in the proof of Theorem \ref{shlp},  $ \prod_{x \in C} C_{x}$ is   compact  when $C_{x}$ is given the weak topology and $ \prod_{x \in C} C_{x}$ is given the corresponding product topology, and $\mathfrak{R}$ is compact in $\prod_{x \in C} C_{x}$. Using the
preorder $\preceq$ in $\mathfrak{R}$ defined by $T \preceq U$ if $q(Tx-Ty)\leq q(Ux-Uy)$ for each $x,y \in C$ and  $q\in Q$,  there exist    a minimal element $P$  in the following sense:

 \text{if }$ T \in \mathfrak{R} \; \text{and } q(Tx-Ty)\leq q(Px-Py)$ \text{ for each  }  $x,y \in C $ \text{and } $q \in Q$ \\
  \text{then } $q(Tx-Ty)= q(Px-Py)$.

  As in the proof of Theorem \ref{shlp}, we have $Px \in \text{Fix}(\sc)$  for every    $x  \in C $,  $P^{2}=P$ and $PT_{t}=T_{t}P=P$.
\end{proof}

\begin{center}
  \bf{ Acknowledgements}
\end{center}  The authors acknowledge the financial support provided by King Mongkut's University of Technology Thonburi through the “KMUTT 55th Anniversary Commemorative Fund”.  This project was supported by the Theoretical and Computational Science (TaCS) Center under Computational and Applied Science for Smart Innovation Cluster (CLASSIC), Faculty of Science, KMUTT.  
Moreover, the first author is  grateful to  the University of Lorestan for their support.

\end{large}
% ------------------------------------------------------------------------
\end{document}